\documentclass{amsart}

\newtheorem{theorem}{Theorem}[section]

\newtheorem{proposition}[theorem]{Proposition}

\theoremstyle{definition}

\theoremstyle{remark}

\newtheorem{claim}{Claim}

\numberwithin{equation}{section}

\newcommand{\abs}[1]{\lvert#1\rvert}

\begin{document}

\title{The fixed point property in a Banach space isomorphic to $c_0$}

\author{Costas Poulios}
\address{Department of Mathematics, University of Athens, 15784, Athens, Greece}
\email{k-poulios@math.uoa.gr}




\date{}

\dedicatory{}

\commby{}

\begin{abstract}
We consider a Banach space, which comes naturally from $c_0$ and it
appears in the literature, and we prove that this space has the
fixed point property for non-expansive mappings.
\end{abstract}

\maketitle

\section{Introduction}\label{sec.introduction}
Let $K$ be a weakly compact, convex subset of a Banach space $X$. A
mapping $T:K\to K$ is called \emph{non-expansive} if
$\left\|Tx-Ty\right\|\le\left\|x-y\right\|$ for any $x,y\in K$. In
the case where every non-expansive map $T:K\to K$ has a fixed point,
we say that $K$ has the \emph{fixed point property}. The space $X$
is said to have the fixed point property if every weakly compact,
convex subset of $X$ has the fixed point property.

A lot of Banach spaces are known to enjoy the aforementioned
property. The earlier results show that uniformly convex spaces have
the fixed point property (see \cite{Br}) and this is also true for
the wider class of spaces with normal structure (see \cite{Kirk}).
The classical Banach spaces $\ell_p, L_p$ with $1<p<\infty$ are
uniformly convex and hence they have the fixed point property. On
the contrary, the space $L_1$ fails this property (see \cite{Als}).

The proofs of many positive results depend on the notion of minimal
invariant sets. Suppose that $K$ is a weakly compact, convex set,
$T:K\to K$ is a non-expansive mapping and $C$ is a nonempty, weakly
compact, convex subset of $K$ such that $T(C)\subseteq C$. The set
$C$ is called \emph{minimal} for $T$ if there is no strictly smaller
weakly compact, convex subset of $C$ which is invariant under $T$. A
straightforward application of Zorn's lemma implies that $K$ always
contains minimal invariant subsets. So, a standard approach in
proving fixed points theorems is to first assume that $K$ itself is
minimal for $T$ and then use the geometrical properties of the space
to show that $K$ must be a singleton. Therefore, $T$ has a fixed
point.

Although a non-expansive map $T:K\to K$ does not have to have fixed
points, it is well-known that $T$ always has an \emph{approximate
fixed point sequence}. This means that there is a sequence $(x_n)$
in $K$ such that $\lim_{n\to\infty} \left\|x_n-Tx_n\right\|=0$. For
such sequences, the following result holds (see \cite{Karl}).

\begin{theorem}\label{th.Karl}
Let $K$ be a weakly compact, convex set in a Banach space, $T:K\to
K$ a non-expansive map, such that $K$ is $T$-minimal, and let
$(x_n)$ be any approximate fixed point sequence. Then, for all $x\in
K$, $$\lim_{n\to \infty}\|x-x_n\|=diam(K).$$
\end{theorem}

Although from the beginning of the theory it became clear that the
classical spaces $\ell_p,L_p$, $1<p<\infty$ have the fixed point
property, the case of $c_0$ remained unsolved for some period of
time. The geometrical properties of this space are not very nice, in
the sense that $c_0$ does not possess normal structure. However, it
was finally proved that the geometry of $c_0$ is still good enough
and and it does not allow the existence of minimal sets with
positive diameter, that is $c_0$ has the fixed point property. This
was done by B. Maurey \cite{Maurey} (see also \cite{Elton}) who also
proved that every reflexive subspace of $L_1$ has the fixed point
property.

\begin{theorem}\label{th.Maurey-1}
The space $c_0$ has the fixed point property.
\end{theorem}

The proof of Theorem \ref{th.Maurey-1} is based on the fact that the
set of approximate fixed point sequences is convex in a natural
sense. More precisely, we have the following (\cite{Maurey},
\cite{Elton}).

\begin{theorem}\label{th.Maurey-2}
Let $K$ be a weakly compact, convex subset of a Banach space which
is minimal for a non-expansive map $T:K\to K$. Let $(x_n)$ and
$(y_n)$ be approximate fixed point sequences for $T$ such that
$\lim_{n\to\infty}\|x_n-y_n\|$ exists. Then there is an approximate
fixed point sequence $(z_n)$ in $K$ such that
$$\lim_{n\to\infty}\|x_n-z_n\|=\lim_{n\to\infty}\|y_n-z_n\|=\frac{1}{2}\lim_{n\to\infty}\|x_n-y_n\|.$$
\end{theorem}

In the present paper, we define a Banach space $X$ isomorphic to
$c_0$ and we prove that this space has the fixed point property. Our
interest on this space derives from several reasons. Firstly, the
space $X$ comes from $c_0$ in a natural way. In fact, the Schauder
basis of $X$ is equivalent to the summing basis of $c_0$. Secondly,
the space $X$ is close to $c_0$ in the sense that the Banach-Mazur
distance between the two spaces is equal to $2$. It is worth
mentioning that from the proof of Theorem \ref{th.Maurey-1} we can
conclude that whenever $Y$ is a Banach space isomorphic to $c_0$ and
the Banach-Mazur distance between $Y$ and $c_0$ is strictly less
than $2$, then $Y$ has the fixed point property. In our case, the
Banach-Mazur distance is equal to $2$, that is the space $X$ lies on
the boundary of what is already known. This fact should also be
compared with the following question in metric fixed point theory:
Find a nontrivial class of Banach spaces invariant under isomorphism
such that each member of the class has the fixed point property (a
trivial example is the class of spaces isomorphic to $\ell_1$). We
shall see that even for spaces close to $c_0$, such as the space
$X$, the situation is quite complicated and this points out the
difficulty of the aforementioned question. Finally, the space $X$
has been used in several places in the study of the geometry of
Banach spaces (for instance see \cite{Hagler}, \cite{Arg}). More
precisely, the well-known Hagler Tree space ($HT$) \cite{Hagler}
contains a plethora of subspaces isomorphic to $X$. Nevertheless, we
do not know if $HT$ has the fixed point property.

\section{Definition and basic properties}\label{sec.definition}
We consider the vector space $c_{00}$ of all real-valued finitely
supported sequences. We let $(e_n)_{n\in\mathbb{N}}$ stand for the
usual unit vector basis of $c_{00}$, that is $e_n(i)=1$ if $i=n$ and
$e_n(i)=0$ if $i\neq n$. If $S\subset \mathbb{N}$ is any
\emph{interval} of integers and $x=(x_i)\in c_{00}$ then we set
$S^\ast(x)=\sum_{i\in S}x_i$. We now define the norm of $x$ as
follows
\begin{equation*}
\|x\|=\sup \abs{S^\ast(x)}
\end{equation*}
where the supremum is taken over all finite intervals
$S\subset\mathbb{N}$. The space $X$ is the completion of the normed
space we have just defined.

It is easily verified that the sequence $(e_n)$ is a normalized
monotone Schauder basis for the space $X$. In the following,
$(e^\ast_n)_{n\in\mathbb{N}}$ denotes the sequence of the
biorthogonal functionals and $(P_n)_{n\in\mathbb{N}}$ denotes the
sequence of the natural projections associated to the basis $(e_n)$.
That is, for any $x=\sum_{i=1}^\infty x_i e_i \in X$ we have
$e_n^\ast(x)=x_n$ and $P_n(x)=\sum_{i=1}^n x_i e_i$.

Furthermore, if $S\subset\mathbb{N}$ is any interval of integers
(not necessarily finite), we define the functional
$S^\ast:X\to\mathbb{R}$ by $S^\ast(x)=S^\ast(\sum_{i=1}^\infty x_i
e_i)=\sum_{i\in S} x_i$. It is easy to see that $S^\ast$ is a
bounded linear functional with $\|S^\ast\|=1$. In the special case
where $S=\mathbb{N}$, the corresponding functional is denoted by
$B^\ast$ (instead of the confusing $\mathbb{N}^\ast$). Therefore,
$B^\ast(x)=\sum_{i=1}^\infty x_i$ for any $x=\sum_{i=1}^\infty x_i
e_i \in X$.

The following proposition provides some useful properties of the
space $X$ and demonstrates the relation between $X$ and $c_0$. We
remind that for any pair $E,F$ of isomorphic normed spaces, the
Banach-Mazur distance between $E$ and $F$ is defined as follows
\begin{equation*}
d(E,F)=\inf\{\|T\|\cdot\|T^{-1}\| \mid T:E\to F ~\text{is an
isomorphism from}~ E ~\text{onto}~ F\}.
\end{equation*}

\begin{proposition}\label{prop.}
The following hold:
\begin{enumerate}
    \item The space $X$ is isomorphic to $c_0$ and in particular the
    basis of $X$ is equivalent to the summing basis of $c_0$.
    \item The subspace of $X^\ast$ generated by the sequence of the
    biorthogonal functionals has codimension $1$. More precisely,
    $X^\ast=\overline{span}\{e_n^\ast\}_{n\in\mathbb{N}}\oplus
    \langle B^\ast \rangle$.
    \item The Banach-Mazur distance $d(X,c_0)$ between $X$ and $c_0$
    is equal to $2$.
\end{enumerate}
\end{proposition}

\begin{proof}
We define the linear operator
\begin{align*}
    \Phi:X&\to c_0 \\
    x=(x_i) &\mapsto \Big(\sum_{i=1}^\infty x_i, \sum_{i=2}^\infty
    x_i, \ldots\Big).
\end{align*}
It is easily verified that $\Phi$ is an isomorphism from $X$ onto
$c_0$ with $\|\Phi\|=1$, $\|\Phi^{-1}\|=2$ and $\Phi$ maps the basis
of $X$ to the summing basis of $c_0$. This proves the first
assertion. The second assertion is an immediate consequence of the
relation between $X$ and $c_0$ established above.

It remains to show that the Banach-Mazur distance $d=d(X,c_0)$ is
equal to $2$. Firstly, we observe that the isomorphism $\Phi$
defined above implies that $d\le 2$. In order to prove the reverse
inequality we fix a real number $\epsilon>0$. Then there exists an
isomorphism $T:X\to c_0$ from $X$ onto $c_0$ such that $\|x\|\le
\|Tx\|_{c_0}\le (d+\epsilon) \|x\|$ for any $x\in X$. We now
consider the normalized sequence $(x_n)$ in $X$ where
$x_n=(x_n(i))_{i\in\mathbb{N}}$ is defined by
\begin{equation*}
x_n(2n-1)=-1, ~ x_n(2n)=1, ~ x_n(i)=0 ~\text{otherwise}.
\end{equation*}
The description of $X^\ast$ given by the second assertion implies
that any bounded sequence $(t_n)_{n\in\mathbb{N}}$ of elements of
$X$ converges weakly to $0$ if and only if $e_m^\ast(t_n)\to 0$ for
every $m\in\mathbb{N}$ and $B^\ast(t_n)\to 0$. It follows that the
sequence $(x_n)_{n\in\mathbb{N}}$ defined above is weakly null. Now
we set $y_n=T(x_n)$ for any $n\in\mathbb{N}$ and we have $1\le
\|y_n\|_{c_0}\le d+\epsilon$ and $(y_n)_{n\in\mathbb{N}}$ converges
weakly to $0$. Therefore, we find $k_1\in\mathbb{N}$ such that the
vectors $y_1$ and $y_{k_1}$ have essentially disjoint supports. More
precisely, since $y_1\in c_0$, there exists $N_1\in\mathbb{N}$ such
that $\abs{y_1(i)}<\epsilon$ for any $i>N_1$. Since $y_n\to 0$
weakly, we find $k_1$ so that $\abs{y_{k_1}(i)}<\epsilon$ for any
$i\le N_1$. It follows that $\|y_1-y_{k_1}\|_{c_0}\leq \max
\{\|y_1\|_{c_0}, \|y_{k_1}\|_{c_0} \} +\epsilon \le d+2\epsilon$. On
the  other hand, $\|x_1-x_{k_1}\|=2$. Therefore,
\begin{equation*}
2=\|x_1-x_{k_1}\| \le \|y_1-y_{k_1}\|_{c_0}\leq d+2\epsilon.
\end{equation*}
If $\epsilon$ tends to $0$, we obtain $2\le d$ as we desire.
\end{proof}

\section{The fixed point property}\label{sec.fpp}
This section is entirely devoted to the proof of the fixed point
property for the space $X$. First we need to establish some
notation. If $S,S'\subset \mathbb{N}$ are intervals we write $S<S'$
to mean that $\max S<\min S'$. Moreover, if $k\in\mathbb{N}$, we
write $k<S$ (resp., $S<k$) to mean $k<\min S$ (resp., $\max S<k$).
Finally, for any $x=(x_i)\in X$, $\text{supp}(x)=\{i\in\mathbb{N}
\mid x_i\neq 0\}$ denotes the support of $x$.

\begin{theorem}\label{th.fpp}
The space $X$ has the fixed point property.
\end{theorem}

\begin{proof}
We follow the standard approach. We assume that $K$ is a weakly
compact, convex subset of $X$ which is minimal for a non-expansive
map $T:K\to K$. Using the geometry of the space $X$, we have to show
that $K$ is a singleton, that is $diam(K)=0$. Let us suppose that
$diam(K)>0$ and now we have to reach a contradiction. Without loss
of generality we may assume that $diam(K)=1$.

Let $(x_n)_{n\in\mathbb{N}}$ be an approximate fixed point sequence
for the map $T$ in the set $K$. By passing to a subsequence and then
using some translation, we may assume that $0\in K$ and $(x_n)$
converges weakly to $0$. Theorem \ref{th.Karl} implies that
$\lim_n\|x_n\|=diam(K)=1$.

We next find a subsequence $(x_{q_n})$ of $(x_n)$ and integers
$l_0=0<l_1<l_2<\ldots$ such that for every $n\in\mathbb{N}$,
$\|P_{l_{n-1}}(x_{q_n})\|<1/n$ and
$\|x_{q_n}-P_{l_n}(x_{q_n})\|<1/n$. The desired sequences
$(x_{q_n})$ and $(l_n)$ are constructed inductively. We start with
$x_{q_1}=x_1$ and $l_0=0$. Suppose that $q_1<q_2<\ldots<q_n$ and
$l_0<l_1<\ldots<l_{n-1}$ have been defined. Then there exists
$l_n>l_{n-1}$ such that $\|x_{q_n}-P_{l_n}(x_{q_n})\|<1/n$. Since
$(x_n)$ is weakly null, it follows that $P_m(x_n)\to 0$ for every
$m\in\mathbb{N}$. Therefore, there exists $q_{n+1}>q_n$ such that
$\|P_{l_n}(x_{q_{n+1}})\|<\frac{1}{n+1}$. The construction of
$(x_{q_n})$ and $(l_n)$ is complete.

Consequently, passing to a subsequence, we may assume that for the
original sequence $(x_n)$ there are integers $l_0=0<l_1<l_2<\ldots$
such that for every $n\in\mathbb{N}$,
\begin{equation*}
\|P_{l_{n-1}}(x_n)\|<\frac{1}{n} ~\text{and}~
\|x_n-P_{l_n}(x_n)\|<\frac{1}{n}.
\end{equation*}
As a matter of fact, we can go one step further and suppose that for
any $n\in\mathbb{N}$, $P_{l_{n-1}}(x_n)=0$ and $x_n-P_{l_n}(x_n)=0$.
Therefore, $\text{supp}(x_n) \subset (l_{n-1},l_n]$, that is $(x_n)$
is a block basis of $(e_n)$. If we did not adopt this assumption,
then in each inequality written below we would have to add a term
equal to $O(\frac{1}{n})$, which simply would change nothing.

We next consider the subsequences $(z_n)=(x_{2n-1})$ and
$(y_n)=(x_{2n})$ and we also set $l_{2n-1}=k_n$ and $l_{2n}=m_n$ for
every $n\in\mathbb{N}$. The properties of the sequence $(x_n)$ imply
that the following hold.
\begin{enumerate}
    \item $(z_n)$ and $(y_n)$ are approximate fixed point sequences
    for the map $T$. Therefore, by Theorem \ref{th.Karl},
    $\lim\|z_n\|=\lim\|y_n\|=1$.
    \item $(z_n)$ and $(y_n)$ converge weakly to $0$.
    \item $\text{supp}(z_n)\subset (m_{n-1},k_n]$ and $\text{supp}(y_n)\subset
    (k_n,m_n]$ for every $n\in\mathbb{N}$.
    \item $\lim\|z_n-y_n\|=1$.
\end{enumerate}
In order to justify the fourth conclusion, we first observe that
$\lim\|z_n-y_n\|\le diam(K)=1$. On the other hand, by the definition
of the norm of the space $X$, for every $n\in\mathbb{N}$ there
exists a finite interval $E_n\subset \mathbb{N}$ such that
$\|z_n\|=\abs{E_n^\ast(z_n)}$. Clearly we may assume that
$E_n\subset(m_{n-1},k_n]$. Then $\|z_n-y_n\|\geq
\abs{E_n^\ast(z_n-y_n)}=\|z_n\|$ and therefore $\lim\|z_n-y_n\|\ge
\lim\|z_n\|=1$.

We are ready now to apply Maurey's theorem (Theorem
\ref{th.Maurey-2}). To this end, we fix a positive integer
$N\in\mathbb{N}$, which will be chosen properly at the end of the
proof, and we set $\epsilon=2^{-N}$. After $N$ iterated applications
of Theorem \ref{th.Maurey-2} we find a sequence $(v_n)$ in the set
$K$ such that: $(v_n)$ is an approximate fixed point sequence for
the map $T$ (which implies that $\lim\|v_n\|=1$) and further
$\lim\|v_n-z_n\|=\epsilon$ and $\lim\|v_n-y_n\|=1-\epsilon$.
Therefore, for all sufficiently large $n\in\mathbb{N}$ the following
hold:
\begin{enumerate}
    \item $\|v_n\|>1-\frac{\epsilon}{2}$;
    \item $\|v_n-z_n\| <3\epsilon/2$ and
    $\|v_n-y_n\|<1-\frac{\epsilon}{2}$;
    \item $\abs{B^\ast(z_n)}<\epsilon/2$ (since $(z_n)$ is weakly
    null).
\end{enumerate}
We also set $S_n=(m_{n-1},k_n]$ so that we have $S_1<S_2<\ldots$.
Concerning the sequence $(v_n)$ in the set $K$ and the sequence of
intervals $(S_n)$ we prove the following two claims.

\begin{claim}\label{claim-first}
For all sufficiently large $n$, the support of $v_n$ is essentially
contained in the interval $S_n$, in the sense that if $S$ is any
interval with $S\cap S_n=\emptyset$ then
$\abs{S^\ast(v_n)}<3\epsilon/2$.
\end{claim}
Indeed, we know that $\text{supp}(z_n)\subset (m_{n-1},k_n]=S_n$.
Therefore, if $S$ is any interval with $S\cap S_n=\emptyset$ then
$S^\ast(z_n)=0$ and hence
\begin{equation*}
\abs{S^\ast(v_n)}= \abs{S^\ast(v_n-z_n)} \leq \|v_n-z_n\|
<\frac{3\epsilon}{2}.
\end{equation*}

\begin{claim}\label{claim-second}
For all sufficiently large $n$, there exist intervals $L_n<R_n$ such
that $S_n=L_n\cup R_n$ and $L_n^\ast(v_n)< -1+7\epsilon$,
$R_n^\ast(v_n)> 1-2\epsilon$.
\end{claim}
We fix a sufficiently large positive integer $n$. Since
$\|v_n\|>1-\frac{\epsilon}{2}$, it follows that there exists a
finite interval $F_n\subset \mathbb{N}$ such that
$\abs{F_n^\ast(v_n)}>1-\frac{\epsilon}{2}$. If $k_n<F_n$, we know by
the previous claim that $\abs{F_n^\ast(v_n)}<3\epsilon/2$, which is
a contradiction. Moreover, if we assume that $F_n\le k_n$ then
$F_n\cap (k_n,m_n]=\emptyset$ and the choice of $(y_n)$ implies
$F_n^\ast(y_n)=0$. Thus,
\begin{equation*}
\abs{F_n^\ast(v_n)} = \abs{F_n^\ast(v_n-y_n)} \leq \|v_n-y_n\|<
1-\frac{\epsilon}{2},
\end{equation*}
which is also a contradiction. By this discussion it is clear that
$\min F_n \le k_n < \max F_n$. Now we set $R_n=F_n\cap [1,k_n]$ and
we estimate
\begin{equation*}
 1-\frac{\epsilon}{2}< \abs{F_n^\ast(v_n)} \le \abs{R_n^\ast(v_n)}
 + \abs{(F_n\setminus R_n)^\ast(v_n)} < \abs{R_n^\ast(v_n)} +
 \frac{3\epsilon}{2},
\end{equation*}
where the last inequality follows by Claim \ref{claim-first}.
Therefore, $\abs{R_n^\ast(v_n)}>1-2\epsilon$. Passing to a
subsequence, we may assume that either $R_n^\ast(v_n)>1-2\epsilon$
for all sufficiently large $n$ or $R_n^\ast(v_n)<-1+2\epsilon$ for
all sufficiently large $n$. We suppose that the first possibility
happens, as the second one is treated similarly (interchanging the
roles of $L_n$ and $R_n$). Consequently, for the interval $R_n$ we
have $\max R_n=k_n$ and $R_n^\ast(v_n)>1-2\epsilon$.

On the other hand, we observe that
\begin{equation*}
\abs{B^\ast(v_n)} \le \abs{B^\ast(v_n-z_n)} + \abs{B^\ast(z_n)} \leq
\|v_n-z_n\|+\frac{\epsilon}{2} <2\epsilon.
\end{equation*}
We note that the sequence $(v_n)$ is not necessarily weakly null.
However, $v_n$ is close to $z_n$ and hence $\abs{B^\ast(v_n)}$ is
very small. We next set $G_n=[1,\min R_n)$ (possibly empty) and
$W_n=(k_n, +\infty)$. Then,
\begin{align*}
    2\epsilon > \abs{B^\ast(v_n)} &=
    \abs{G_n^\ast(v_n)+R_n^\ast(v_n)+W_n^\ast(v_n)}\\
    &\ge R_n^\ast(v_n)- \abs{G_n^\ast(v_n)}-\abs{W_n^\ast(v_n)}\\
    &> 1-2\epsilon - \abs{G_n^\ast(v_n)} -
    \frac{3\epsilon}{2}.
\end{align*}
Therefore $G_n$ is non-empty and $\abs{G_n^\ast(v_n)}> 1-
\frac{11\epsilon}{2}$. However, if $G_n^\ast(v_n)>1-
\frac{11\epsilon}{2}$, then it would follow
\begin{equation*}
\abs{B^\ast(v_n)} \ge R_n^\ast(v_n) + G_n^\ast(v_n)-
\abs{W_n^\ast(v_n)} \geq 2-9\epsilon,
\end{equation*}
which is a contradiction. Hence, $G_n^\ast(v_n)<
-1+\frac{11\epsilon}{2}$. Further, we observe that we can not have
$G_n<S_n$, since in this case it would follow
$\abs{G_n^\ast(v_n)}<\frac{3\epsilon}{2}$. Consequently, $\max
G_n>m_{n-1}$ which clearly implies $\min R_n>m_{n-1}+1$. Finally, we
set $L_n=G_n\cap (m_{n-1},k_n]$ and we estimate
\begin{equation*}
    -1+\frac{11\epsilon}{2} > G_n^\ast(v_n) = L_n^\ast(v_n) +(G_n\setminus
    L_n)^\ast(v_n) \ge L_n^\ast(v_n) - \frac{3\epsilon}{2}.
\end{equation*}
We deduce that $L_n^\ast(v_n) <-1+7\epsilon$. Therefore, the
intervals $L_n<R_n$ satisfy the following: $S_n=L_n\cup R_n$,
$R_n^\ast(v_n)> 1-2\epsilon$ and $L_n^\ast(v_n) <-1+7\epsilon$. The
proof of the claim is now complete.

Using the construction and the properties of the sequences $(v_n)$
and $(S_n)$, we can reach the final contradiction and finish the
proof of the theorem. Indeed, we fix a sufficiently large
$n\in\mathbb{N}$ and we consider the intervals $D=(k_n,m_n]$ and
$S=R_n\cup D\cup L_{n+1}$. Then, using Claim \ref{claim-first} and
Claim \ref{claim-second}  we have
\begin{align*}
    S^\ast(v_n)&= R_n^\ast(v_n) + (D\cup L_{n+1})^\ast(v_n) >
    1-2\epsilon - \frac{3\epsilon}{2} =
    1-\frac{7\epsilon}{2}\\
    S^\ast(v_{n+1}) &= (R_n\cup D)^\ast(v_{n+1})+
    L_{n+1}^\ast(v_{n+1}) < \frac{3\epsilon}{2} -1 +
    7\epsilon =-1 + \frac{17\epsilon}{2}.
\end{align*}
Therefore,
\begin{equation*}
\|v_n-v_{n+1}\| \ge \abs{S^\ast(v_n-v_{n+1})} =
\abs{S^\ast(v_n)-S^\ast(v_{n+1})} \ge 2-12\epsilon.
\end{equation*}
If $\epsilon$ has been chosen small enough, then we have a
contradiction, since $\|v_n-v_{n+1}\| \le diam(K)=1$.
\end{proof}

\bibliographystyle{amsplain}

\end{document}